\documentclass[12pt]{amsart}

\usepackage{amssymb}
\usepackage{amsmath}
\usepackage{amsthm}

\usepackage{color}

\usepackage[pagebackref,hypertexnames=false, colorlinks, citecolor=red, linkcolor=red]{hyperref} 
\usepackage[backrefs]{amsrefs}

\allowdisplaybreaks

\def\rr{{\mathbb R}}

\def\fz{\infty}

\def\lz{\lambda}

\def\tbz{{\Delta_\lz}}
\def\dmz{{dm_\lz}}

\def\rrp{{{\mathbb\rr}_+}}

\def\inzf{{\int_0^\fz}}

\def\lpz{{L^p(\rrp,\, dm_\lz)}}

\newtheorem{thm}{Theorem}[section]
\newtheorem{lem}[thm]{Lemma}

\numberwithin{equation}{section}

\begin{document}

\arraycolsep=1pt

\title{The Two-Weight Inequality for the Poisson Operator in the Bessel Setting}

\author[J. Li]{Ji Li}
\address{Ji Li, Department of Mathematics\\ Macquarie University\\ NSW, 2109, Australia.}
\email{ji.li@mq.edu.au}
\thanks{J. Li's research supported by ARC DP 160100153.}

\author[B. D. Wick]{Brett D. Wick}
\address{Brett D. Wick, Department of Mathematics\\ Washington University -- St. Louis\\ One Brookings Drive\\ St. Louis, MO USA 63130-4899}
\email{wick@math.wustl.edu}
\thanks{B. D. Wick's research supported in part by National Science Foundation DMS grants \#0955432 and \#1560955.}

\subjclass[2010]{Primary: 42B20}
\keywords{Bessel operator}

\date{\today}

\begin{abstract}
Fix $\lambda>0$. Consider the Bessel operator $\Delta_\lambda:=-\frac{d^2}{dx^2}-\frac{2\lambda}{x}\frac d{dx}$ on $\mathbb{R}_+:=(0,\infty)$ and the harmonic conjugacy introduced by Muckenhoupt and Stein. We provide the two-weight inequality for the Poisson operator $\mathsf{P}^{[\lambda]}_t=e^{-t\sqrt{\Delta_\lambda}}$  in this Bessel setting.  In particular, we prove that for a measure $\mu$ on $\mathbb{R}^2_{+,+}:=(0,\infty)\times (0,\infty)$ and $\sigma$ on $\mathbb{R}_+$:
$$
\|\mathsf{P}^{[\lambda]}_\sigma(f)\|_{L^2(\mathbb{R}^2_{+,+};\mu)} \lesssim \|f\|_{L^2(\mathbb{R}_+;\sigma)},
$$
if and only if testing conditions hold for the the Poisson operator and its adjoint.  Further, the norm of the operator is shown to be equivalent to the best constant in the testing conditions.
\end{abstract}

\maketitle

\section{Introduction and Statement of Main Results\label{s1}}

The aim of this paper is to provide the necessary and sufficient conditions for a two-weight inequality for the Poisson operators in the Bessel setting.  The theory of classical harmonic analysis is considered to be intimately connected to the Laplacian; changing the differential operator introduces new challenges and directions to explore.
In 1965, Muckenhoupt and Stein in \cite{ms} introduced a notion of conjugacy associated with this Bessel operator $\tbz$,
which is defined by
\begin{equation*}
\tbz f(x):=-\frac{d^2}{dx^2}f(x)-\frac{2\lz}{x}\frac{d}{dx}f(x),\quad x>0.
\end{equation*}
They developed a theory in the setting of
$\tbz$ which parallels the classical one associated to $\Delta$.  For $p\in[1, \fz)$, $\rr_+:=(0, \fz)$ and $\dmz(x):= x^{2\lz}\,dx$ results on $\lpz$-boundedness of conjugate
functions and fractional integrals associated with $\tbz$ were
obtained.  Since then, many problems based on the Bessel context were studied;
see, for example, \cites{ak,bcfr,bfbmt,bfs,bhnv,k78,v08,yy}. In particular,
the properties and $L^p$ boundedness $(1<p<\infty)$ of Riesz transforms
\begin{equation*}
R_{\Delta_\lambda}f := \partial_x (\Delta_\lz)^{-\frac{1}{2}}f
\end{equation*}
related to $\Delta_\lz$ have been studied in \cites{ak,bcfr,bfbmt,ms,v08}.

Next we recall the Poisson integral, the conjugate Poisson integral in the Bessel setting.
As in \cite{bdt}, let $\left\{\mathsf{P}^{[\lz]}_t\right\}_{t>0}$ be the Poisson semigroup $\left\{e^{-t\sqrt{\Delta_\lz}}\right\}_{t>0}$ defined by
\begin{equation*}
\mathsf{P}^{[\lz]}_tf(x):=\inzf \mathsf{P}^{[\lz]}_t(x,y)f(y)y^{2\lz}\,dy,
\end{equation*}
where
\begin{equation*}
\mathsf{P}^{[\lz]}_t(x,y)=\inzf e^{-tz}(xz)^{-\lz+\frac{1}{2}}J_{\lz-\frac{1}{2}}(xz)(yz)^{-\lz+\frac{1}{2}}J_{\lz-\frac{1}{2}}(yz)z^{2\lz}\,dz
\end{equation*}
and $J_\nu$ is the Bessel function of the first kind and order $\nu$.  Weinstein \cite{w48} established the following formula for $\mathsf{P}^{[\lz]}_t(x,y)$: $t,\,x,\, y\in\mathbb{R}_+$,
\begin{equation*}
\mathsf{P}^{[\lz]}_t(x,y)=\frac{2\lz t}{\pi}\int_0^\pi\frac{(\sin\theta)^{2\lz-1}}{(x^2+y^2+t^2-2xy\cos\theta)^{\lz+1}}\,d\theta.
\end{equation*}

Let  $\sigma$ be a weight on $\mathbb{R}_+:=(0,\infty)$ and $\mu$ be a weight on $\mathbb{R}^2_{+,+}:=(0,\infty)\times (0,\infty)$.
Consider the inequality
\begin{align}\label{two weight}
\|\mathsf{P}^{[\lz]}_\sigma(f)\|_{L^2(\mathbb{R}^2_{+,+};\mu)} \leq \mathcal{N}\|f\|_{L^2(\mathbb{R}_+;\sigma)},
\end{align}
where
\begin{align*}
\mathsf{P}^{[\lz]}_\sigma(f)(x,t):= \inzf \mathsf{P}^{[\lz]}_t(x,y)f(y) \,d\sigma(y).
\end{align*}
We use $\mathsf{P}^{[\lz],*}_\mu$ to denote the dual operator of $\mathsf{P}^{\lz}$, defined as follows
\begin{align*}
\left\langle \mathsf{P}^{[\lz]}_\sigma(f), g \right\rangle_{L^2(\mathbb{R}_{+,+}^2;\mu)} &=\int_{\mathbb{R}^2_{+,+}} \inzf \mathsf{P}^{[\lz]}_t(x,y)f(y) \,d\sigma(y) g(x,t)\,d\mu(x,t)\\
&= \inzf   \int_{\mathbb{R}^2_{+,+}}  \mathsf{P}^{[\lz]}_t(x,y)g(x,t)\,d\mu(x,t)   f(y) \,d\sigma(y)\\
&= \inzf  \mathsf{P}^{[\lz],*}_\mu(g)(y)f(y) \,d\sigma(y).
\end{align*}
So in particular,
\begin{align*}
\mathsf{P}^{[\lz],*}_\mu(g)(y):= \int_{\mathbb{R}^2_{+,+}} \mathsf{P}^{[\lz]}_t(x,y)g(x,t) \,d\mu(x,t).
\end{align*}
We also observe that a simple duality argument provides for:
\begin{align}\label{two weight dual}
\left\Vert \mathsf{P}^{[\lz],*}_\mu(\phi)\right\Vert_{L^2(\mathbb{R}_+;\sigma)} \lesssim \mathcal{N}\left\Vert \phi\right\Vert_{L^2(\mathbb{R}^2_{+,+};\mu)}
\end{align}

The main result of this paper is the following two-weight inequality for the Poisson operator $\{\mathsf{P}^{\lz}_t\}_{t>0}$. 
\begin{thm}\label{main theorem}
Let $\sigma$ be a measure on $\mathbb{R}_+$ and $\mu$ a measure on $\mathbb{R}_{+,+}^2$.
The following conditions are equivalent:
\begin{enumerate}
\item[(1)] The two-weight inequality \eqref{two weight} holds.  Namely,
\begin{align*}
\|\mathsf{P}^{[\lz]}_\sigma(f)\|_{L^2(\mathbb{R}^2_{+,+};\mu)} \leq \mathcal{N}\|f\|_{L^2(\mathbb{R}_+,\sigma)};
\end{align*}
\item[(2)] The testing conditions below hold uniformly over all intervals $I\subset (0,\infty)$
\begin{align*}
\int_{\hat{3I}} \mathsf{P}^{[\lz]}_\sigma(1_I)(x,t)^2 \,d\mu(x,t) & \leq \mathcal{F}^2 \sigma(I),\\
\int_{3I} \mathsf{P}^{[\lz],*}_\mu(t1_{\hat{I}})(y)^2 \,d\sigma(y) & \leq \mathcal{B}^2 \int_{\hat{I}} t^2 \,d\mu(x,t).
\end{align*}
\end{enumerate}
Moreover, we have that $\mathcal{N}\simeq \mathcal{F}+\mathcal{B}$.
Here $1_I$ is the indicator of $I$, $\hat{I}= I\times [ 0, |I| ]$.  
\end{thm}
It is immediate that the testing conditions are necessary and that $\mathcal{F}+\mathcal{B}\lesssim \mathcal{N}$.  The forward condition follows by testing \eqref{two weight} on an indicator function and restricting the region of integration.  The backward condition follows by testing the dual inequality \eqref{two weight dual} on the indicator of a set and then again restricting the integration.  In the remainder of the paper we address how to show that these testing conditions are sufficient to prove \eqref{two weight} and \eqref{two weight dual}.  In the course of the proof we will also demonstrate that $\mathcal{N}\lesssim \mathcal{F}+\mathcal{B}$.

Below we use the notation $X\lesssim Y$ to denote that there is an absolute constant $C$ so that $X\leq CY$.  If we write $X\approx Y$, then we mean that $X\lesssim Y$ and $Y\lesssim X$.  And, $:=$ means equal by definition.



\section{Proof of the Two-Weight Inequality for  \texorpdfstring{$\left\{\mathsf{P}^{[\lz]}_t\right\}_{t>0}$}{the Bessel Poisson Operator}}
\label{s:MainResult}

We now prove that the testing conditions imply the norm inequality for the Poisson operator $\mathsf{P}^{[\lz]}_t$ in the Bessel setting, following the line of Sawyer's original argument, \cite{s88}, and using some modification contained in the proof given by Lacey, \cite{L}.

To begin with, we assume that $\sigma$ is restricted to some large dyadic interval $I_0\subset \mathbb{R}_+$, and  that $\mu$ is restricted to $3\hat{I}_0$.  There is no loss in assuming that the measures $\sigma$ and $\mu$ are compactly supported since the resulting estimates will not depend upon the support in any way, and we can then pass to the general case through a standard limiting argument.

To prove \eqref{two weight}, by duality, it suffices to prove that for every $\phi\in L^2(\mathbb{R}^2_{+,+})$,
\begin{align}\label{two weight p}
\int_{\mathbb{R}_+} | \mathsf{P}^{[\lz],*}_\mu(\phi)(x)|^2 \,d\sigma(x) \lesssim \left(\mathcal{F}^2+\mathcal{B}^2\right)  
\int_{\mathbb{R}^2_{+,+}} | \phi(y,t)|^2 \,d\mu(y,t).
\end{align}
Before we continue, we point out that it suffices to prove \eqref{two weight p}
for $\phi\in C^\infty(\mathbb{R}^2_{+,+})$.
And before
 the decomposition of $\int_{\mathbb{R}_+} | \mathsf{P}^{[\lz],*}_\mu(\phi)(x)|^2 \,d\sigma(x) $ as in the left-hand side of the targeted inequality above, we first claim that
\begin{align}\label{claim finite}
\int_{\mathbb{R}_+} | \mathsf{P}^{[\lz],*}_\mu(\phi)(x)|^2 \,d\sigma(x) <\infty.
\end{align}
In fact, since $\sigma$ is restricted to  $I_0$ and $\mu$ is restricted to $3\hat{I}_0$, we see that
\begin{align*}
&\int_{\mathbb{R}_+} | \mathsf{P}^{[\lz],*}_\mu(\phi)(x)|^2 \,d\sigma(x) \\
&=\int_{I_0} \bigg| \int_{3\hat{I}_0}  \mathsf{P}^{[\lz]}_t(y,x)\phi(y,t)\,d\mu(y,t)  \bigg|^2 \,d\sigma(x) \\
&=\int_{I_0} \bigg| \int_{3\hat{I}_0}  \frac{2\lz t}{\pi}\int_0^\pi\frac{(\sin\theta)^{2\lz-1}}{(y^2+x^2+t^2-2yx\cos\theta)^{\lz+1}}\,d\theta\ \phi(y,t)\,d\mu(y,t)  \bigg|^2 \,d\sigma(x) \\
&\leq C\int_{I_0} \mu(3\hat{I}_0) \int_{3\hat{I}_0} \bigg[{1\over m_\lambda(I(y,t)) + m_\lambda(I(y,|x-y|))}\bigg( {t\over t+|x-y|} \bigg)^\gamma\bigg]^2\\
&\qquad\qquad\times |\phi(y,t)|^2\,d\mu(y,t) \,d\sigma(x)\\
&<\infty,
\end{align*}
where the first inequality follows from H\"older's inequality and the pointwise estimate of the Poisson kernel $ \mathsf{P}^{[\lz]}_t(y,x)$ contained in \cite{yy} with the constant $\gamma>0$, and the last inequality follows from 
the fact that $\phi \in C^\infty(\mathbb{R}^2_{+,+})$, i.e., $\phi$ has compact support $G$ contained in 
$\mathbb{R}^2_{+,+}$, which implies that for every $(x,t)\in G$, $t\geq t_0>0$.
Thus, we obtain that the claim \eqref{claim finite} holds.

Take a non-negative function $\phi \in L^2 (\mathbb{R}^2_{+,+};\mu)$. Consider the open sets
$$ \Omega_k:=\left\{ x\in\mathbb{R}_+: \mathsf{P}^{[\lz],*}_\mu(\phi)(x)  >2^k \right\}. $$
Let $\mathcal{I}_k$ be a Whitney decomposition of $\Omega_k$. Namely, an interval $I\in \mathcal{D}$ is in 
$\mathcal{I}_k$ if and only if $I$ is maximal (in the sense of set containment) subject to the conditions: $(3I \cap \mathbb{R}_+)\subset \Omega_k$ and 
$(5I\cap \mathbb{R}_+)\not\subset\Omega_k$.  Here and throughout the whole paper, 
we use $\mathcal D$ to denote the system of dyadic intervals in $\mathbb R_+$.

We first show that these $\mathcal{I}_k$ are well-defined, i.e., there exist such maximal interval $I$ satisfying $(3I \cap \mathbb{R}_+)\subset \Omega_k$ and 
$(5I\cap \mathbb{R}_+)\not\subset\Omega_k$. Recall that a general interval $I(x,r)$ centered at $x$ with radius $r$ in the Bessel setting is defined as 
\begin{align}\label{interval} 
I(x,r)=(x-r,x+r) \cap \mathbb{R}_+.
\end{align}

Now choose an arbitrary dyadic interval $I\subset \Omega_k$ such that $3I\subset \Omega_k$. Consider the interval $5I$. We have the following 3 cases:

\smallskip

Case 1:  $3I\subset  \mathbb{R}_+$ and $5I \subset  \mathbb{R}_+$.

If $5I\cap \Omega_k^c \not=\emptyset$, then we put this $I$ in $\mathcal{I}_k$.  If $5I\subset \Omega_k$, then we consider the father of $I$, denote it by $\tilde{I}$. 

\smallskip
Case 2:  $3I\subset  \mathbb{R}_+$ and $5I \cap  (-\infty,0) \not=\emptyset$.

In this case, the dyadic interval $I$ must be the first one after the dyadic interval $(0, |I|)$ in the level of length $I$. 
Thus, based on the definition of general intervals in \eqref{interval}, $5I=(0,4|I|)$.
If $5I\cap \Omega_k^c \not=\emptyset$, then we put this $I$ in $\mathcal{I}_k$.  If $5I\subset \Omega_k$, then we consider the father of $I$, denote it by $\tilde{I}$. And in fact, $\tilde{I}= (0,2|I|)$. 

\smallskip
Case 3:  $3I  \cap  (-\infty,0) \not=\emptyset$ and $5I \cap  (-\infty,0) \not=\emptyset$.

In this case, the dyadic interval $I$ must be $(0, |I|)$.
Thus, based on the definition of general intervals in \eqref{interval}, $5I=(0,3|I|)$.
If $5I\cap \Omega_k^c \not=\emptyset$, then we put this $I$ in $\mathcal{I}_k$.  If $5I\subset \Omega_k$, then we consider the father of $I$, denote it by $\tilde{I}$. And in fact, $\tilde{I}= (0,2|I|)$. 
\smallskip

Combining all these three cases, we get that, if $I$ is not the maximal one, i.e. $5I\subset \Omega_k$, then we can further consider $\tilde{I}$, the father of $I$. And we can deduce that
\begin{align*}
\tilde{I} &\subset \Omega_k, {\rm\ since\ } \tilde{I}\subset 3I {\rm\ and\ } 3I\subset \Omega_k;\\
3\tilde{I} &\subset \Omega_k, {\rm\ since\ } 3\tilde{I}\subset 5I {\rm\ and\ } 5I\subset \Omega_k.
\end{align*}
Hence $\tilde{I} $ is the next right candidate, and it suffices to consider $5\tilde{I}$. By induction, we can always obtain a maximal dyadic interval $J$ subject to $(3J \cap \mathbb{R}_+)\subset \Omega_k$ and 
$(5J\cap \mathbb{R}_+)\not\subset\Omega_k$.

We further point out that these collections $\mathcal{I}_k$  satisfy the following properties:
\begin{enumerate}
\item[(1)] $\Omega_k=\cup_{I\in\mathcal{I}_k} I$, and the intervals $I\in\mathcal{I}_k$ are either equal or disjoint, except the endpoints;
\item[(2)] $3I\subset \Omega_k$ and 
$5I\not\subset\Omega_k$;
\item[(3)] If $I\in\mathcal{I}_k$ and $I'\in\mathcal{I}_{k'}$ with $I\subsetneq I'$, then $2^k > 2^{k'}$;
\item[(4)] For all $k\in\mathbb{Z}$, $\sum_{I\in\mathcal{I}_k} 1_{3I}(x) \leq C 1_{\Omega_k}$.
\end{enumerate}
\smallskip
Let $m$ be a large constant to determined later. Then from the Whitney decomposition, we have that 
the left-hand side of \eqref{two weight p} is bounded as follows.
\begin{align*}
\int_{\mathbb{R}_+} | \mathsf{P}^{[\lz],*}_\mu(\phi)(y)|^2 \,d\sigma(y) 
&= \sum_{k\in\mathbb{Z}}\int_{\Omega_{k+m}\backslash\Omega_{k+m+1}} | \mathsf{P}^{[\lz],*}_\mu(\phi)(y)|^2 \,d\sigma(y) \\
&\leq  \sum_{k\in\mathbb{Z}} 2^{2k}\sigma(\Omega_{k+m}\backslash\Omega_{k+m+1})\\
&=\sum_{k\in\mathbb{Z}} 2^{2k} \sum_{I\in\mathcal{I}_k}\sigma(I\cap (\Omega_{k+m}\backslash\Omega_{k+m+1})).
\end{align*}
We denote $F_k(I):= I \cap (\Omega_{k+m}\backslash\Omega_{k+m+1})$. Now let $\delta\in(0,1)$, to be chosen sufficiently small.  Then we have
\begin{align*}
\int_{\mathbb{R}_+} | \mathsf{P}^{[\lz],*}_\mu(\phi)(y)|^2 \,d\sigma(y) 
&=\sum_{k\in\mathbb{Z}} 2^{2k} \sum_{\substack{I\in\mathcal{I}_k\\ \sigma(F_k(I))<\delta \sigma(I) }}\sigma(F_k(I))
+\sum_{k\in\mathbb{Z}} 2^{2k} \sum_{\substack{I\in\mathcal{I}_k\\ \sigma(F_k(I))\geq\delta \sigma(I) }}\sigma(F_k(I))\\
&=:A+B.
\end{align*}
As for the term $A$, it is obvious that 
\begin{align*}
A&\leq \delta\sum_{k\in\mathbb{Z}} 2^{2k} \sum_{\substack{I\in\mathcal{I}_k\\ \sigma(F_k(I))<\delta \sigma(I) }}\sigma(I) \leq  \delta\sum_{k\in\mathbb{Z}} 2^{2k} \sigma(\Omega_k) \leq  \delta \int_{\mathbb{R}_+} | \mathsf{P}^{[\lz],*}_\mu(\phi)(y)|^2 \,d\sigma(y),
\end{align*}
which will be absorbed into the left-hand side provided that $\delta$ is sufficiently small, based on the priori condition as in \eqref{claim finite}.

Thus, to prove \eqref{two weight p}, it remains to show that term $B$ can be dominated in terms of the testing conditions, i.e.,
\begin{equation}\label{B}
B\lesssim \left(\mathcal{F}^2+\mathcal{B}^2\right)\left\Vert \phi\right\Vert_{L^2(\mathbb{R}_{+,+}^2;\mu)}^2.
\end{equation}

To continue, we first show that the Poisson operator $\mathsf{P}^{[\lz]}_{t}$ satisfies the following maximum principle.
\begin{lem}\label{lem2.1}
There exists a positive constant $C$ such that
\begin{align}\label{Pmu*}
\mathsf{P}^{[\lz],*}_\mu(\phi\cdot 1_{(3\hat{I})^c} )(x) < 19^{\lz+1}2^k
\end{align}
for all $x\in I$, $I\in\mathcal{I}_k$ and $k\in\mathbb{Z}$.
\end{lem}
\begin{proof}
Note that $I$ is the Whitney interval, satisfying $3I\subset \Omega_k$ and 
$5I\not\subset\Omega_k$. We now choose $z\in 5I\cap \Omega_k^c$. Then we obtain that
$|I|<|z-x|<3|I|$.
Now we claim the following:  For $z\in 5I\cap \Omega_k^c$ and for every $y$ with $(y,t)\not\in 3\hat{I}$,
there holds
\begin{align}\label{claim P1}
\mathsf{P}^{[\lz]}_t(x,y)\leq 19^{\lz+1}\mathsf{P}^{[\lz]}_t(z,y).
\end{align}
Assume this claim holds. Then we multiply it by $\phi(y,t) 1_{(3\hat{I})^c}$ and then integrate with respect
to $d\mu(y,t)$. As a consequence, we have
\begin{align*}
\mathsf{P}^{[\lz],*}_\mu(\phi\cdot 1_{(3\hat{I})^c} )(x)\leq 19^{\lz+1} \mathsf{P}^{[\lz],*}_\mu(\phi\cdot 1_{(3\hat{I})^c} )(z) \leq 19^{\lz+1}2^k,
\end{align*}
which implies that \eqref{Pmu*} holds.

It now suffices to prove this claim, we consider the following two cases.

\noindent
\textbf{Case 1: $y\not\in 3I$. } In this case we have $|y-x|>|I|$. For each $y\not\in 3I$ and for each fixed $\theta\in (0,\pi)$,
we denote $d_x^2:=x^2+y^2-2xy\cos\theta$, and $d_z^2:=z^2+y^2-2zy\cos\theta$. Then by the triangle inequality, we have $d_x>|x-y|$, which implies that
$d_x>|I|$.  Next, we obtain that 
$$ d_z<d_x+|z-x|<d_x+3|I| < 4d_x. $$
As a consequence,
\begin{align*}
\mathsf{P}^{[\lz]}_t(x,y)&=\frac{2\lz t}{\pi}\int_0^\pi\frac{(\sin\theta)^{2\lz-1}}{(x^2+y^2+t^2-2xy\cos\theta)^{\lz+1}}\,d\theta\\
&\leq 16^{\lz+1}\frac{2\lz t}{\pi}\int_0^\pi\frac{(\sin\theta)^{2\lz-1}}{(z^2+y^2+t^2-2zy\cos\theta)^{\lz+1}}\,d\theta\\
&\leq 16^{\lz+1}\mathsf{P}^{[\lz]}_t(z,y).
\end{align*}

\noindent
\textbf{Case 2: $y\in 3I$. } Since we require that $(y,t)\not\in 3\hat{I}$, in this case we have $t>|I|$.
Then from the triangle inequality, we have
$$ d_z<d_x+|z-x|<d_x+3|I| < d_x+3t, $$
which implies that $t^2+d_z^2 <19(d_x^2+t^2)$. As a consequence, we find 
\begin{align*}
\mathsf{P}^{[\lz]}_t(x,y)\leq 19^{\lz+1}\mathsf{P}^{[\lz]}_t(z,y).
\end{align*}
Combining the above two cases, we obtain that the claim \eqref{claim P1} holds.

The proof of Lemma \ref{lem2.1} is complete.
\end{proof}

Now for $I\in\mathcal{I}_k$ with $\sigma(F_k(I))\geq\delta \sigma(I) $ and for each $x\in F_k(I)$, it follows from
the maximum principle that
\begin{align*}
\mathsf{P}^{[\lz],*}_\mu(\phi\cdot 1_{3\hat{I}} )(x) &=
\mathsf{P}^{[\lz],*}_\mu(\phi )(x) -  \mathsf{P}^{[\lz],*}_\mu(\phi\cdot 1_{(3\hat{I})^c} )(x) \\
&\geq 2^{k+m} - 19^{\lz+1}2^k. 
\end{align*}
By choosing $m$ large such that $2^m>19^{\lz+1} +1$, we obtain that 
\begin{align*}
\mathsf{P}^{[\lz],*}_\mu(\phi\cdot 1_{3\hat{I}} )(x) \geq 2^k. 
\end{align*}
Hence,
\begin{align*}
2^k &\leq \frac{1}{\sigma(F_k(I))} \int _{F_k(I)} \mathsf{P}^{[\lz],*}_\mu(\phi\cdot 1_{3\hat{I}} )(x)\,d\sigma(x) =\frac{1}{\sigma(F_k(I))}  \int_{3\hat{I}} \mathsf{P}^{[\lz]}_\sigma( 1_{F_k(I)} )(x,t) \phi(x,t) \,d\mu(x,t)\\
&=\frac{1}{\sigma(F_k(I))}  \int_{3\hat{I}\backslash \hat{\Omega}_{k+m+1} } \mathsf{P}^{[\lz]}_\sigma( 1_{F_k(I)} )(x,t) \phi(x,t) \,d\mu(x,t)\\
&\quad+\frac{1}{\sigma(F_k(I))}  \int_{3\hat{I} \cap  \hat{\Omega}_{k+m+1}} \mathsf{P}^{[\lz]}_\sigma( 1_{F_k(I)} )(x,t) \phi(x,t) \,d\mu(x,t)\\
&=: B_1(k,I)+B_2(k,I).
\end{align*}
Then we obtain that 
\begin{align*}
B&\leq 2\sum_{k\in\mathbb{Z}}  \sum_{\substack{I\in\mathcal{I}_k\\ \sigma(F_k(I))\geq\delta \sigma(I) }}
B_1(k,I)^2\sigma(F_k(I)) +2\sum_{k\in\mathbb{Z}}  \sum_{\substack{I\in\mathcal{I}_k\\ \sigma(F_k(I))\geq\delta \sigma(I) }}
B_2(k,I)^2\sigma(F_k(I))
=:B_1+B_2.
\end{align*}
To prove \eqref{B}, we seek to prove that:
\begin{eqnarray}
B_1 & \lesssim & \delta^{-2}\mathcal{F}^2\left\Vert \phi\right\Vert_{L^2(\mathbb{R}_{+,+}^2;\mu)}^2; \label{e:B1estimate}\\
B_2 & \lesssim & \delta^{-2}\left(\mathcal{F}^2+\mathcal{B}^2\right)\left\Vert \phi\right\Vert_{L^2(\mathbb{R}_{+,+}^2;\mu)}^2; \label{e:B2estimate}
\end{eqnarray}

We now consider the term $B_1$. As for $B_1(k,I)$, by noting that $\sigma(I)\geq \sigma(F_k(I))\geq\delta \sigma(I) $ and that the Poisson operator $\mathsf{P}^{[\lambda]}_{t}$ is a positive operator,
 we have
 \begin{align*}
B_1(k,I)&\leq \delta^{-1}\frac{1}{\sigma(I)}  \int_{3\hat{I}\backslash \hat{\Omega}_{k+m+1} } \mathsf{P}^{[\lz]}_\sigma( 1_{I} )(x,t) \phi(x,t) \,d\mu(x,t)\\
&\leq  \delta^{-1}\frac{1}{\sigma(I)} \left(\int_{3\hat{I}\backslash \hat{\Omega}_{k+m+1} } |\mathsf{P}^{[\lz]}_\sigma( 1_{I} )(x,t)|^2  \,d\mu(x,t)\right)^{\frac{1}{2}}\left(\int_{3\hat{I}\backslash \hat{\Omega}_{k+m+1} } |\phi(x,t)|^2  \,d\mu(x,t)\right)^{\frac{1}{2}}  \\
&\leq \delta^{-1}\mathcal{F} \frac{1}{\sigma(I)^{\frac{1}{2}}} \left(\int_{3\hat{I}\backslash \hat{\Omega}_{k+m+1} } |\phi(x,t)|^2  \,d\mu(x,t)\right)^{\frac{1}{2}},
\end{align*}
where the last inequality follows from the forward testing condition for $\mathsf{P}^{[\lambda]}_{t}$. Hence,
\begin{align*}
B_1&\leq 2\delta^{-2}\mathcal{F}^2\sum_{k\in\mathbb{Z}}  \sum_{\substack{I\in\mathcal{I}_k\\ \sigma(F_k(I))\geq\delta \sigma(I) }}
 \frac{1}{\sigma(I)} \int_{3\hat{I}\backslash \hat{\Omega}_{k+m+1} } |\phi(x,t)|^2  \,d\mu(x,t)\sigma(F_k(I))\\
&\leq 2\delta^{-2}\mathcal{F}^2
  \int_{\mathbb{R}^2_{+,+}} |\phi(x,t)|^2 \sum_{k\in\mathbb{Z}}  \sum_{\substack{I\in\mathcal{I}_k\\ \sigma(F_k(I))\geq\delta \sigma(I) }} 
  1_{3\hat{I}\backslash \hat{\Omega}_{k+m+1}}(x,t)  \,d\mu(x,t)\\
&\lesssim \delta^{-2}\mathcal{F}^2
  \int_{\mathbb{R}^2_{+,+}} |\phi(x,t)|^2   \,d\mu(x,t),
\end{align*}
where the last inequality follows from the fact that
$$ \left\|\sum_{k\in\mathbb{Z}}  \sum_{\substack{I\in\mathcal{I}_k\\ \sigma(F_k(I))\geq\delta \sigma(I) }} 
  1_{3\hat{I}\backslash \hat{\Omega}_{k+m+1}}(x,t)\right\|_{\infty} \lesssim 1, $$
which is a consequence of the bounded overlaps of the Whitney cubes.  Thus, we have that  $B_1\lesssim \delta^{-2}\mathcal{F}^2 \left\Vert \phi\right\Vert_{L^2(\mathbb{R}_{+,+}^2;\mu)}^2$, which shows that  \eqref{e:B1estimate} holds.

We now estimate $B_2$, which is bounded by
\begin{align}\label{B2 e1}
B_2&\leq 2\delta^{-1}\sum_{k\in\mathbb{Z}}  \sum_{\substack{I\in\mathcal{I}_k\\ \sigma(F_k(I))\geq\delta \sigma(I) }}
\frac{1}{\sigma(I)}  \left(\int_{3\hat{I} \cap  \hat{\Omega}_{k+m+1}} \mathsf{P}^{[\lz]}_\sigma( 1_{F_k(I)} )(x,t) \phi(x,t) \,d\mu(x,t) \right)^2. 
\end{align}
To continue, we decompose 
\begin{align}\label{B2 decom}
 3\hat{I}\cap  \hat{\Omega}_{k+m+1}=\bigcup_{J} \{ \hat{J}: J\subset 3I, J\in \mathcal{I}_{k+m+1}  \}.
\end{align}
Note that for such $J$, $3J\cap F_k(I)\not=\emptyset$. 
We now claim that or $(x,t)\in\hat{J}$,
\begin{align}\label{claim P}
\mathsf{P}^{[\lambda]}_\sigma(1_{F_k(I)})(x,t) \approx \frac{t}{|J|}\mathsf{P}^{[\lambda]}_\sigma(1_{F_k(I)})(x_J,|J|), 
\end{align}
where the implicit constants are independent of $x$, $t$ and $I$.

In fact,  for $(x,t)\in\hat{J}$ and $y\in F_k(I)$, we have $|x-y|>|J|>t$. Moreover, for such $x$ and $y$ and for $\theta\in(0,\pi)$, we have
$d=\sqrt{x^2+y^2-2xy\cos\theta} >|x-y|$, which yields that $d>|J|>t$. As a consequence, we have
\begin{align*}
\mathsf{P}^{[\lambda]}_\sigma(1_{F_k(I)})(x,t)&= \frac{t}{|J|} \int_{F_k(I)} \frac{2\lz |J|}{\pi}\int_0^\pi\frac{(\sin\theta)^{2\lz-1}}{(x^2+y^2+t^2-2xy\cos\theta)^{\lz+1}}\,d\theta \,d\sigma(y)\\
& \approx \frac{t}{|J|} \int_{F_k(I)} \frac{2\lz |J|}{\pi}\int_0^\pi\frac{(\sin\theta)^{2\lz-1}}{(x^2+y^2+|J|^2-2xy\cos\theta)^{\lz+1}}\,d\theta \,d\sigma(y)\\
& =\frac{t}{|J|}\mathsf{P}^{[\lambda]}_\sigma(1_{F_k(I)})(x,|J|).
\end{align*}
Next we denote $d_J=\sqrt{x_J^2+y^2-2x_Jy\cos\theta}$. Note that $|x-x_J|<|J|/2$. We have
$$  d-|x-x_J| < d_J < d_x+|x-x_J|, $$
which implies that $ d/2 <d_J < 3d_J/2 $. Hence, we get
\begin{align*}
& \int_{F_k(I)} \frac{2\lz |J|}{\pi}\int_0^\pi\frac{(\sin\theta)^{2\lz-1}}{(x^2+y^2+|J|^2-2xy\cos\theta)^{\lz+1}}\,d\theta \,d\sigma(y)\\
&\approx \int_{F_k(I)} \frac{2\lz |J|}{\pi}\int_0^\pi\frac{(\sin\theta)^{2\lz-1}}{(x_J^2+y^2+|J|^2-2x_Jy\cos\theta)^{\lz+1}}\,d\theta \,d\sigma(y),
\end{align*}
which implies that 
$$ \mathsf{P}^{[\lambda]}_\sigma(1_{F_k(I)})(x,|J|)\approx \mathsf{P}^{[\lambda]}_\sigma(1_{F_k(I)})(x_J,|J|). $$
Thus, the claim \eqref{claim P} holds.

From \eqref{claim P} we obtain that
\begin{align}
&\int_{\hat{J}}\mathsf{P}^{[\lambda]}_\sigma(1_{F_k(I)})(x,t)\phi(x,t)\,d\mu(x,t)\nonumber \\
&\approx  
\mathsf{P}^{[\lambda]}_\sigma(1_{F_k(I)})(x_J,|J|) \int_{\hat{J}} \frac{t}{|J|} \phi(x,t)\,d\mu(x,t)\nonumber \\
&\approx  
\mathsf{P}^{[\lambda]}_\sigma(1_{F_k(I)})(x_J,|J|) \int_{\hat{J}} \frac{1}{t|J|} \phi(x,t)\,d\tilde{\mu}(x,t) \nonumber\\
&\approx  
\int_{\hat{J}}  \mathsf{P}^{[\lambda]}_\sigma(1_{F_k(I)})(x,t)  \,d\tilde{\mu}(x,t) \cdot \frac{1}{\tilde{\mu}(\hat{J})} \cdot \frac{1}{|J|}\int_{\hat{J}} \frac{1}{t} \phi(x,t)\,d\tilde{\mu}(x,t) \nonumber\\
&\lesssim \int_{\hat{J}}  \mathsf{P}^{[\lambda]}_\sigma(1_{I})(x,t) \frac{1}{t} d\tilde{\mu}(x,t) \cdot \frac{1}{\tilde{\mu}(\hat{J})} \cdot\int_{\hat{J}} \frac{1}{t} \phi(x,t)\,d\tilde{\mu}(x,t), \label{B2 e2}
\end{align}
where the last inequality follows from the fact that $\mathsf{P}^{[\lambda]}_t$ is a positive operator, and $d\tilde{\mu}(x,t) =
t^2\,d\mu(x,t)$.

From \eqref{B2 e1}, the decomposition \eqref{B2 decom} and the inequality \eqref{B2 e2}, we get that
\begin{align*}
B_2&\leq C\delta^{-1}\sum_{k\in\mathbb{Z}}  \sum_{\substack{I\in\mathcal{I}_k\\ \sigma(F_k(I))\geq\delta \sigma(I) }}\frac{1}{\sigma(I)}  \left(\sum_{\substack{ J\in \mathcal{I}_{k+m+1}\\ J\subset 3I }} \int_{\hat{J}}  \mathsf{P}^{[\lambda]}_\sigma(1_{I})(x,t) \,\frac{d\tilde{\mu}(x,t)}{t}  \cdot \frac{\int_{\hat{J}} \frac{1}{t} \phi(x,t)d\tilde{\mu}(x,t)}{\tilde{\mu}(\hat{J})}  \right)^2. 
\end{align*}
We now define $$\alpha(J)= \frac{1}{\tilde{\mu}(\hat{J})} \int_{\hat{J}} \frac{1}{t} \phi(x,t)d\tilde{\mu}(x,t)$$ for every interval $J \subset \mathbb{R}_+$. Since $\phi\in L^2(\mathbb{R}^2_{+,+};\mu)$, we have
\begin{align*}
\alpha(J)&\leq \left( \frac{1}{\tilde{\mu}(\hat{J})} \int_{\hat{J}} \Big|\frac{1}{t} \phi(x,t)\Big|^2d\tilde{\mu}(x,t)\right)^{\frac{1}{2}}\\
&\leq \left( \frac{1}{\tilde{\mu}(\hat{J})} \int_{\mathbb{R}^2_{+,+}} \Big| \phi(x,t)\Big|^2\,d\mu(x,t)\right)^{\frac{1}{2}}\\
&=\frac{1}{\tilde{\mu}(\hat{J})^{\frac{1}{2}}} \|\phi\|_{L^2(\mathbb{R}^2_{+,+};\mu)}. 
\end{align*}
Hence $\alpha(J)$ is well-defined for each $J$.

We now define the set $\mathcal{G}$ of principal intervals as follows. Initialize $\mathcal{G}$
to be $I_0$, which is the large dyadic interval that $\sigma$ is supported on.  Next, consider the children $J$ of 
$I_0$.  If $\alpha(J) \geq 10 \alpha(I_0)$, then add $J$ to $\mathcal{G}$. If $\alpha(J) < 10 \alpha(I_0)$, then we continue to look at the children of this $J$. Then the set $\mathcal{G}$ is defined via induction.

Next we consider the maximal function 
$$M_{\tilde{\mu}}\psi(x,t) := \sup_{J\in \mathcal{D},\ (x,t)\in \hat{J}} \frac{1}{\tilde{\mu}(\hat{J})} \int_{\hat{J}}  |\psi(y,s)|d\tilde{\mu}(y,s).  $$
\begin{align}\label{claim max}
{\rm We\ {\bf claim}\ that}:  M_{\tilde{\mu}}\psi(x,t)\ {\rm\ is\ bounded\ on\ }\ L^2(\mathbb{R}_{+,+}^2;\tilde{\mu}).   
\end{align}

In fact, it is easy to see that the maximal function $M_{\tilde{\mu}}\psi(x,t)$
is bounded on $L^\infty(\mathbb R^2_{+,+})$. Thus, it suffices to show that 
it is also weak type $(1,1)$.

To see this, let $0\leq \psi\in L^1(\mathbb R^2_{+,+})$ and $\alpha>0$. Consider the level set
$$ S_\alpha:=\{ (x,t)\in\mathbb R_{+,+}^2:\ M_{\tilde{\mu}}\psi(x,t) >\alpha \}, $$
which is the union of the maximal dyadic cubes $\hat J = J \times [0,|J|]$ in $\mathbb R^2_{+,+}$ with some $J\in\mathcal D$ such that
$$ \int_{\hat{J}}  |\psi(y,s)|d\tilde{\mu}(y,s) > \alpha  \tilde{\mu}(\hat{J})>0.  $$
Here, the argument $\hat J $ is maximal means that if there is a $J_1\in\mathcal D$ with $J\subsetneq J_1$, then 
$$ \int_{\hat{J}_1}  |\psi(y,s)|d\tilde{\mu}(y,s) \leq \alpha \tilde{\mu}(\hat{J}_1).  $$
And we point out that such maximal dyadic cubes always exist. In fact, suppose 
there is $(x,t)\in S_\alpha$ such that there is no maximal dyadic cube in those dyadic cubes that contain $(x,t)$.
There we have a sequence of increasing nested dyadic cubes $\hat J_k$ containing $(x,t)$ such that 
$  \tilde{\mu}(\hat{J}_k) \to\infty$ as $k\to\infty$ with 
$$ \int_{\hat{J}_k}  |\psi(y,s)|d\tilde{\mu}(y,s) > \alpha  \tilde{\mu}(\hat{J}_k).  $$
However, this leads to contradiction since $$\int_{\hat{J}_k}  |\psi(y,s)|d\tilde{\mu}(y,s) \to \|\psi\|_{L^1(\mathbb R^2_{+,+})}.$$

Thus, we have a sequence of disjoint dyadic maximal cubes  $\{\hat J_k\}_k$ such that
$$ S_\alpha\subset \bigcup_k \hat J_k. $$
This sequence  $\{\hat J_k\}_k$ can be finite or infinite.
In both cases, we have
\begin{align}
\sum_k \tilde{\mu}(\hat{J}_k)\leq {1\over \alpha} \sum_k  \int_{\hat{J}_k}  |\psi(y,s)|d\tilde{\mu}(y,s) 
\leq {1\over \alpha}  \int_{\mathbb R^2_{+,+}}  |\psi(y,s)|d\tilde{\mu}(y,s) <\infty.
\end{align}
So, if the selection gives an infinite sequence then $\tilde{\mu}(\hat{J}_k)\to 0$ as $k\to\infty$.

As a consequence, we obtain that
$$ \tilde{\mu}(S_\alpha)\leq  \sum_k \tilde{\mu}(\hat{J}_k) \leq {1\over \alpha}\|\psi\|_{L^1(\mathbb R^2_{+,+})},$$
which implies that $M_{\tilde{\mu}}$ is weak type $(1,1)$, and hence \eqref{claim max} holds.


\smallskip
Now from the  $L^2(\mathbb{R}_{+,+}^2;\tilde{\mu})$ bound of $M_{\tilde{\mu}}$, we have
\begin{align}\label{maximal function}
\sum_{I\in\mathcal{G}} \alpha(I)^2 \tilde{\mu}(\hat{I}) &\leq  \sum_{I\in\mathcal{G}} \left( \inf_{(x,t)\in \hat{I}} M_{\tilde{\mu}}(\tilde{\phi})(x,t)\right)^2 \tilde{\mu}(\hat{I}) \\
&\leq \int_{3\hat{I}_0} M_{\tilde{\mu}}(\tilde{\phi})(x,t)^2 \,d\tilde{\mu}(x,t)\nonumber\\
&\lesssim \int_{3\hat{I}_0} \tilde{\phi}(x,t)^2 \,d\tilde{\mu}(x,t)\nonumber\\
&\leq  \|\phi\|_{L^2(\mathbb{R}^2_{+,+},\,d\mu)}^2, \nonumber
\end{align}
where $\tilde{\phi}(x,t) = t^{-1} \phi(x,t)$.
\bigskip

Next, in the sum over $\mathcal{I}_{k+m+1}$, we denote $I_{-1}=I-|I|$, $I_0=I$ and $I_1=I+|I|$. The union of these three intervals is $3I$. This notation,together with the definition of $\mathcal{G}$,  gives
\begin{align*}
B_2&\lesssim \delta^{-1}\sum_{k\in\mathbb{Z}}  \sum_{\substack{I\in\mathcal{I}_k\\ \sigma(F_k(I))\geq\delta \sigma(I) }}
\frac{1}{\sigma(I)} \\
&\quad\quad\times \left(\sum_{\theta=-1}^1 \sum_{\substack{ J\in \mathcal{I}_{k+m+1}\\ J\subset I_\theta, \pi_\mathcal{G}J=\pi_\mathcal{G} I_\theta}} \int_{\hat{J}}  \mathsf{P}^{[\lambda]}_\sigma(1_{I})(x,t) \,\frac{d\tilde{\mu}(x,t)}{t}  \cdot \frac{\int_{\hat{J}} \frac{1}{t} \phi(x,t)d\tilde{\mu}(x,t)}{\tilde{\mu}(\hat{J})}  \right)^2\\
&\quad\quad \quad+\delta^{-1}\sum_{k\in\mathbb{Z}}  \sum_{\substack{I\in\mathcal{I}_k\\ \sigma(F_k(I))\geq\delta \sigma(I) }}
\frac{1}{\sigma(I)} \\
&\quad\quad\quad\quad\times \left(\sum_{\theta=-1}^1 \sum_{\substack{ J\in \mathcal{I}_{k+m+1}\\ J\subset I_\theta, \pi_\mathcal{G}J\subsetneq\pi_\mathcal{G} I_\theta}} \int_{\hat{J}}  \mathsf{P}^{[\lambda]}_\sigma(1_{I})(x,t) \frac{1}{t} \,d\tilde{\mu}(x,t) \cdot \frac{\int_{\hat{J}} \frac{1}{t} \phi(x,t)\,d\tilde{\mu}(x,t) }{\tilde{\mu}(\hat{J})} \right)^2 \\
&=: B_{21}+B_{22}.
\end{align*}
Thus, to prove \eqref{e:B2estimate} it will suffice to provide an estimate of the right form on each of $B_{21}$ and $B_{22}$.  We will show that:
\begin{eqnarray}
B_{21} & \lesssim & \delta^{-2} \mathcal{B}^2  \|\phi\|_{L^2(\mathbb{R}^2_{+,+};\mu)}^2;\label{e:B21estimate}\\
B_{22} & \lesssim & \delta^{-1}\mathcal{F}^2\|\phi\|_{L^2(\mathbb{R}^2_{+,+};\mu)}^2.\label{e:B22estimate}
\end{eqnarray}

For $B_{21}$, using the definition of $\alpha(J)$, we have
\begin{align*}
B_{21}&\lesssim \delta^{-1}\sum_{\theta=-1}^1\sum_{k\in\mathbb{Z}}  \sum_{\substack{I\in\mathcal{I}_k\\ \sigma(F_k(I))\geq\delta \sigma(I) }}
\frac{1}{\sigma(I)}  \alpha(J)^2\left( \sum_{\substack{ J\in \mathcal{I}_{k+m+1}\\ J\subset I_\theta, \pi_\mathcal{G}J=\pi_\mathcal{G} I_\theta}} \int_{\hat{J}}  \mathsf{P}^{[\lambda]}_\sigma(1_{I})(x,t) \,\frac{d\tilde{\mu}(x,t)}{t}   \right)^2\\
&\lesssim \delta^{-1}\sum_{\theta=-1}^1\sum_{k\in\mathbb{Z}}  \sum_{\substack{I\in\mathcal{I}_k\\ \sigma(F_k(I))\geq\delta \sigma(I) }}
\frac{1}{\sigma(I)}  \alpha(\pi_{\mathcal{G}}I_\theta)^2\left( \sum_{\substack{ J\in \mathcal{I}_{k+m+1}\\ J\subset I_\theta, \pi_\mathcal{G}J=\pi_\mathcal{G} I_\theta}} \int_{\hat{J}}  \mathsf{P}^{[\lambda]}_\sigma(1_{I})(x,t) t \,d\mu(x,t)  \right)^2\\
&\lesssim \delta^{-1}\sum_{\theta=-1}^1\sum_{k\in\mathbb{Z}}  \sum_{\substack{I\in\mathcal{I}_k\\ \sigma(F_k(I))\geq\delta \sigma(I) }}
\frac{1}{\sigma(I)}  \alpha(\pi_{\mathcal{G}}I_\theta)^2\left( \int_{I}  \mathsf{P}^{[\lambda],*}_\mu(t 1_{\hat{I_\theta}})(y)  \,d\sigma(y)  \right)^2\\
&\lesssim \delta^{-1}\sum_{\theta=-1}^1\sum_{k\in\mathbb{Z}}  \sum_{\substack{I\in\mathcal{I}_k\\ \sigma(F_k(I))\geq\delta \sigma(I) }}
\frac{1}{\sigma(I)}  \alpha(\pi_{\mathcal{G}}I_\theta)^2\, \sigma(I)   \int_{3I}  \mathsf{P}^{[\lambda],*}_\mu(t 1_{\hat{I_\theta}})(y)^2  \,d\sigma(y)    \\
&\lesssim\delta^{-1}\mathcal{B}^2\sum_{\theta=-1}^1\sum_{k\in\mathbb{Z}}  \sum_{\substack{I\in\mathcal{I}_k\\ \sigma(F_k(I))\geq\delta \sigma(I) }}
  \alpha(\pi_{\mathcal{G}}I_\theta)^2\,    \tilde{\mu}(\hat{I}_\theta)   \\
&= \delta^{-1}\mathcal{B}^2\sum_{\theta=-1}^1  \sum_{G\in\mathcal{G}}  \alpha(G)^2  \sum_{k\in\mathbb{Z}}  \sum_{\substack{I\in\mathcal{I}_k\\ \sigma(F_k(I))\geq\delta \sigma(I)\\  \pi_{\mathcal{G}} I_\theta =G}}
  \tilde{\mu}(\hat{I}_\theta),
\end{align*}
where the last inequality follows from the testing condition for $\mathsf{P}^{[\lambda],*}_{t}$.

We point out that for each dyadic interval $I$, the set
$$ \left\{k\in\mathbb{Z}:\ I\in\mathcal{I}_k, \sigma(F_k(I)) \geq \delta \sigma(I)\right\} $$
consists of at most $\delta^{-1}$ consecutive integers. Actually, that the integers
in this set are consecutive follows from the nested property of the collections $\mathcal{I}_k$.
Moreover, note that for each fixed $I$, the sets $F_k(I)\subset I$ are pairwise disjoint (with respect to $k$), and 
for each $k$,  $\sigma(F_k(I)) \geq \delta \sigma(I)$. Hence, there are at most $\delta^{-1}$ such integers $k$.

As a consequence, we obtain that
\begin{align*}
B_{21}&\leq C\delta^{-2} \mathcal{B}^2\sum_{G\in\mathcal{G}}  \alpha(G)^2  
  \tilde{\mu}(\hat{G})\leq C\delta^{-2} \mathcal{B}^2  \|\phi\|_{L^2(\mathbb{R}^2_{+,+};\mu)}^2,
\end{align*}
where the last inequality follows from the maximal inequality \eqref{maximal function}.  This gives \eqref{e:B21estimate}.

We now turn to the estimate $B_{22}$. Using the definition of $\alpha(J)$, we have
\begin{align*}
B_{22}&\lesssim\delta^{-1}\sum_{\theta=-1}^1 \sum_{k\in\mathbb{Z}}  \sum_{\substack{I\in\mathcal{I}_k\\ \sigma(F_k(I))\geq\delta \sigma(I) }}
\frac{1}{\sigma(I)}  \left(\sum_{\substack{ J\in \mathcal{I}_{k+m+1}\\ J\subset I_\theta\\ \pi_\mathcal{G}J\subsetneq\pi_\mathcal{G} I_\theta}} \int_{\hat{J}}  \mathsf{P}^{[\lambda]}_\sigma(1_{I})(x,t) \,\frac{d\tilde{\mu}(x,t)}{t}  \cdot \alpha(J) \right)^2 \\
&\lesssim\delta^{-1}\sum_{\theta=-1}^1 \sum_{k\in\mathbb{Z}}  \sum_{\substack{I\in\mathcal{I}_k\\ \sigma(F_k(I))\geq\delta \sigma(I) }}
\frac{1}{\sigma(I)}  \sum_{\substack{ J\in \mathcal{I}_{k+m+1}\\ J\subset I_\theta\\ \pi_\mathcal{G}J\subsetneq\pi_\mathcal{G} I_\theta}} \left[\int_{\hat{J}}  \mathsf{P}^{[\lambda]}_\sigma(1_{I})(x,t)\, \frac{d\tilde{\mu}(x,t)}{t} \right]^{2}\tilde{\mu}(\hat{J})^{-1}  \\
&\quad\quad\times  \sum_{\substack{ J\in \mathcal{I}_{k+m+1}\\ J\subset I_\theta\\ \pi_\mathcal{G}J\subsetneq\pi_\mathcal{G} I_\theta}} \tilde{\mu}(\hat{J}) \alpha(J)^2,
\end{align*}
where the last inequality follows from Cauchy--Schwarz inequality. Next, from the Cauchy--Schwarz inequality, the definition of $\tilde{\mu}$ and the testing condition, we have
\begin{align*}
  \sum_{\substack{ J\in \mathcal{I}_{k+m+1}\\ J\subset I_\theta\\ \pi_\mathcal{G}J\subsetneq\pi_\mathcal{G} I_\theta}} \left[\int_{\hat{J}}  \mathsf{P}^{[\lambda]}_\sigma(1_{I})(x,t) \,\frac{d\tilde{\mu}(x,t)}{t} \right]^{2}\tilde{\mu}(\hat{J})^{-1}   & \leq \sum_{\substack{ J\in \mathcal{I}_{k+m+1}\\ J\subset I_\theta\\ \pi_\mathcal{G}J\subsetneq\pi_\mathcal{G} I_\theta}} \int_{\hat{J}}  \mathsf{P}^{[\lambda]}_\sigma(1_{I})(x,t)^2  \,d\mu(x,t)\\
& \leq \mathcal{F}^2 \sigma(I),
\end{align*}
which implies that
\begin{align}
B_{22}
&\lesssim\mathcal{F}^2\delta^{-1}\sum_{\theta=-1}^1 \sum_{k\in\mathbb{Z}}  \sum_{\substack{I\in\mathcal{I}_k\\ \sigma(F_k(I))\geq\delta \sigma(I) }}
\frac{1}{\sigma(I)}   \sigma(I)  \sum_{\substack{ J\in \mathcal{I}_{k+m+1}\\ J\subset I_\theta\\ \pi_\mathcal{G}J\subsetneq\pi_\mathcal{G} I_\theta}} \tilde{\mu}(\hat{J}) \alpha(J)^2\nonumber\\
&\lesssim\mathcal{F}^2\delta^{-1} \sum_{k\in\mathbb{Z}}  \sum_{\substack{I\in\mathcal{I}_k\\ \sigma(F_k(I))\geq\delta \sigma(I) }}
\sum_{\substack{ J\in \mathcal{I}_{k+m+1}\\ J\subset I_\theta, \pi_\mathcal{G}J\subsetneq\pi_\mathcal{G} I_\theta}} \tilde{\mu}(\hat{J}) \alpha(\pi_\mathcal{G}J)^2.\label{B22}
\end{align}

We now recall a technical result from Lacey \cite{L}*{Lemma 8.15} as follows.
\begin{lem}[\cite{L}]\label{L}
There is an absolute constant $C$ such that for any $G\in\mathcal{G}$, the cardinality of the set
$$ \left\{k:\ \pi_{\mathcal{G}}J = G,\ J\in\mathcal{I}_{k+m+1}\ {\rm contributes\ to\ the\ } k{\rm th\ sum\ in\ } \eqref{B22} \right\} $$
is at most $C$.
\end{lem}
As a consequence of Lemma \ref{L}, we get that 
\begin{align*}
B_{22}
&\lesssim \mathcal{F}^2\delta^{-1} \sum_{I\in \mathcal{G}} \tilde{\mu}(\hat{I}) \alpha(I)^2 \lesssim\mathcal{F}^2\delta^{-1}  \|\phi\|_{L^2(\mathbb{R}^2_{+,+};\mu)}^2,
\end{align*}
which gives \eqref{e:B22estimate}.

Thus, both \eqref{e:B21estimate} and \eqref{e:B22estimate} implies that \eqref{e:B2estimate} holds, which together with \eqref{e:B1estimate}, shows that \eqref{B} holds. Hence, \eqref{two weight p} holds, which shows that Theorem \ref{main theorem} holds.

\end{document}